\documentclass[12pt,a4paper]{amsart}

\title{A Note on Element Centralizers in Finite Coxeter Groups.}

\author{Matja\v{z} Konvalinka}
\address{M.K.: 
Department of Mathematics,
Vanderbilt University, Nashville TN, USA}
\email{matjaz.konvalinka@vanderbilt.edu}

\author{G\"otz Pfeiffer}
\author{Claas E. R\"over}
\address{G.P., C.E.R.: 
School of Mathematics, Statistics and Applied Mathematics,
NUI Galway,
University Road, Galway, Ireland}
\email{goetz.pfeiffer@nuigalway.ie, claas.roever@nuigalway.ie}

\usepackage{geometry}
\usepackage[colorlinks=true]{hyperref}
\usepackage{amssymb} 
\let\emptyset\varnothing

\newcommand{\Size}[1]{\left|#1\right|}
\newcommand{\Span}[1]{\left<#1\right>}

\newcommand{\Sym}{\mathfrak{S}}

\let\AA\relax
\newcommand{\AA}{\mathcal{A}}
\newcommand{\DD}{\mathcal{D}}

\newcommand{\N}{\mathbb{N}}
\newcommand{\R}{\mathbb{R}}

\renewcommand{\labelenumi}{(\roman{enumi})}

\newcommand{\Fix}{\mathop{\mathrm{Fix}}}
\newcommand{\diag}{\mathop{\mathrm{diag}}}

\newcommand{\partition}{$ be a partition of $}
\newcommand{\dpartition}{$ be a double partition of $}

\newtheorem{Theorem}{Theorem}[section]
\newtheorem{Proposition}[Theorem]{Proposition}

\newtheorem{Lemma}[Theorem]{Lemma}

\theoremstyle{definition}
\newtheorem{Definition}[Theorem]{Definition}

\numberwithin{equation}{section}

\begin{document}

\begin{abstract}
  The normalizer $N_W(W_J)$ of a standard parabolic subgroup $W_J$ of
  a finite Coxeter group $W$ splits over the parabolic subgroup with
  complement $N_J$ consisting of certain minimal length coset
  representatives of $W_J$ in $W$.  In this note we show that (with
  the exception of a small number of cases arising from a situation in
  Coxeter groups of type $D_n$) the centralizer $C_W(w)$ of an element
  $w \in W$ is in a similar way a semidirect product of the
  centralizer of $w$ in a suitable small parabolic subgroup $W_J$ with
  complement isomorphic to the normalizer complement~$N_J$. 
  Then we use this result to give a new short proof of Solomon's 
  Character Formula and discuss its connection to MacMahon master theorem.
\end{abstract}


\maketitle

\section{Introduction}
\label{sec:intro}

Let $W$ be a finite Coxeter group, generated by a set of simple
reflections $S$ with length function $\ell \colon W \to \N\cup\{0\}$.  Each
subset $J \subseteq S$ generates a 
so-called standard parabolic
subgroup $W_J = \Span{J}$ of $W$.
Conjugates of standard parabolic subgroups are called parabolic
subgroups.  These subgroups are themselves Coxeter groups and
therefore play an important role in the structure theory of finite
Coxeter groups.  A well-known property of the cosets of a standard
parabolic subgroup $W_J$ in $W$ is that each coset contains a unique
element of minimal length.  The subgroup $W_J$ hence
possesses a
distinguished right transversal $X_J$, consisting of the
minimal length coset representatives.
Due to a theorem of Howlett~\cite{Howlett1980} and later work of Brink
and Howlett~\cite{BrinkHowlett1999}, it is known that and how the
normalizer $N_W(W_J)$ of the parabolic subgroup $W_J$ is a semidirect
product of $W_J$ and a subgroup $N_J$ consisting of precisely those
minimal length coset representatives $x \in X_J$ which leave $J$ as
subset of $W$ invariant in the conjugation action of $W$ on its
subsets, i.e., $N_J = \{x \in X_J: J^x = J\}$.

In this note we show that most centralizers of elements in $W$ enjoy a
similar semidirect product decomposition.  Pfeiffer and
R\"ohrle~\cite{PfeifferRoehrle2005} have shown, based on
Richardson's~\cite{Richardson1982} characterization of involutions as
central longest elements of parabolic subgroups of $W$, that if $w \in
W$ is an involution then its centralizer in $W$ coincides with the
normalizer of a parabolic subgroup, and as such is a semidirect
product.  This note can be regarded as a generalization of the result
for involutions to all elements of~$W$.  Our results effectively
reduce questions regarding the conjugacy classes of elements in a
finite Coxeter group $W$ to the cuspidal conjugacy classes, that is
those conjugacy classes which are disjoint from any proper parabolic
subgroup of~$W$.  Cuspidal conjugacy classes of elements of $W$ play a
central role in the algorithmic approach to the conjugacy classes of
finite Coxeter groups in Chapter~3 of the book by Geck and
Pfeiffer~\cite{GePf2000}.  We refer the reader to this book as a
general introduction to the theory of finite Coxeter groups.

We will call certain conjugacy classes of elements of a finite Coxeter
group $W$ \emph{non-compliant}; see Definition~\ref{def:non-compliant}.
  Without exception, these are
conjugacy classes of $W$ which nontrivially intersect a parabolic
subgroup of type $D_n$ with $n > 4$.  Hence, if $W$ has no parabolic
subgroups of type $D$, part~(ii) of the following theorem applies
without restrictions. We can now formulate our main theorem as follows.

\begin{Theorem} \label{thm:1} Let $W$ be a finite Coxeter group and
  let $w \in W$.  Let $V$ be the smallest parabolic subgroup of $W$
  that contains $w$.  Then the following hold.
\begin{enumerate}
\item The centralizer $C_V(w) = C_W(w) \cap V$ is a normal
  subgroup of the centralizer $C_W(w)$ with quotient $C_W(w)/C_V(w)$
  isomorphic to the normalizer quotient $N_W(V)/V$.
\item The centralizer $C_W(w)$ splits over $C_V(w)$ with complement
  isomorphic to $N_W(V)/V$ unless $w$ lies in a non-compliant conjugacy
  class of elements of~$W$.
\end{enumerate}
\end{Theorem}

The parabolic subgroup $V$ in the theorem is well-defined as the
intersection of all parabolic subgroups of $W$ that contain $w$, due
to the fact that intersections of parabolic subgroups are parabolic
subgroups, see Theorem~\ref{2.1.12} below.  For the proof of
the theorem, we will assume that $w$ has \emph{minimal length} in its
conjugacy class in $W$.  Then $V$ is the \emph{standard} parabolic
subgroup $W_J$ of $W$, where $J = J(w)$, the set of generators
occurring in a reduced expression for $w$.  The proof of part~(i) is
carried out in Section~\ref{sec:centralizers}.  Part~(ii) of the
theorem is established case by case in
Section~\ref{sec:cases}.  The results for the exceptional types of
Coxeter groups have been obtained with the help of computer programs
using the GAP~\cite{GAP} package CHEVIE~\cite{chevie}.  These programs
are available through the second author's ZigZag~\cite{zigzag} package.
In Section~\ref{sec:application}, we use Theorem~\ref{thm:1} to
provide a new short proof of a theorem of Solomon, and then discuss
its relation to MacMahon master theorem~\cite[page 98]{macmahon}. 

\section{Preliminaries.}
\label{sec:prelim}

In this section we recall some results about distinguished coset
representatives and conjugacy classes in a finite Coxeter group $W$,
generated by a set of simple reflections $S$ and with length function
$\ell$.  

For $w \in W$, we set $J(w) = \{s_1, \dots, s_l\} \subseteq S$, if $w
= s_1 \dots s_l$ is a reduced expression, i.e., if $l = \ell(w)$.
As a consequence of Matsumoto's theorem, $J(w)$ does not depend
on the choice of a reduced expression for~$w$.

For $w \in W$, let 
\begin{align*}
  \DD(w) = \{s \in S : \ell(sw) < \ell(w)\}
\end{align*}
be its \emph{descent set}, and
let 
\begin{align*}
   \AA(w) = \{s \in S : \ell(sw) > \ell(w)\} = S \setminus \DD(w)
\end{align*}
be its \emph{ascent set}.
The set 
\begin{align*}
  X_J = \{ w \in W : J \subseteq \AA(w) \}
\end{align*}
is a right transversal for $W_J$ in $W$, consisting of the
elements of minimal length in each coset.  For each element $w \in W$
there are unique elements $u \in W_J$ and $x \in X_J$ such that $w = u
\cdot x$.  Here the explicit multiplication dot indicates that the
product $ux$ is \emph{reduced}, i.e., that $\ell(ux) = \ell(u) +
\ell(x)$.  An immediate consequence is the following lemma.

\begin{Lemma}[\protect{\cite[Lemma 2.1.14]{GePf2000}}] 
  \label{2.1.14}
  Let $J \subseteq S$.  Then $\ell(w^x) \geq \ell(w)$ for all $w \in
  W_J$, $x \in X_J$.
\end{Lemma}

We denote the longest element of $W$ by $w_0$.
For $J \subseteq  S$, we denote by $w_J$ the  longest element of the
parabolic subgroup $W_J$. 

\begin{Lemma}
  \label{1.5.2}
  Let $w \in W$ and let $J = \DD(w)$.  Then $w = w_J \cdot x$ for some
  $x \in X_J$.  
\end{Lemma}

\begin{proof}
  This follows from \cite[Lemma 1.5.2]{GePf2000} and
\cite[Proposition 2.1.1]{GePf2000}
\end{proof}

For $J,K\subseteq  S$ define  $X_{JK} = X_J \cap  X_K^{-1}$.  Then
$X_{JK}$ is  a set of  minimal length double coset  representatives of
$W_J$ and $W_K$ in $W$.

\begin{Theorem}[\protect{\cite[Theorem 2.1.12]{GePf2000}}] 
  \label{2.1.12}
  Let $J, K \subseteq S$ and  let $x \in X_{JK}$. Then $W_J^x \cap W_K
  = W_L$, where $L = J^x \cap K$.
\end{Theorem}

\begin{Theorem}[\protect{\cite[Theorem 2.3.3]{GePf2000}}] 
  \label{2.3.3}
  Suppose $J, K$ are conjugate subsets of $S$ and that $x \in X_J$ is
  such that $J^x = K$.  If $s \in \DD(x)$ then $x = d \cdot y$, where
  $d = w_J w_L$ for $L= J \cup \{s\}$, and $y \in X_L$.
\end{Theorem}



For the conjugacy classes of $W$, we are particularly interested in
elements of minimal length.  These elements have useful properties,
such as the following.

\begin{Proposition}[\protect{\cite[Corollary 3.1.11]{GePf2000}}] 
  \label{3.1.11}
  Let $C$ be a conjugcay class of $W$ and let $w, w'$ be elements of
  minimal length in $C$.  Then $J(w') = J(w)^x$ for some $x \in
  X_{J(w), J(w')}$.
\end{Proposition}

A conjugacy class $C$ of elements of $W$ is called a \emph{cuspidal class} if $C \cap W_J = \emptyset$ for all proper subsets $J$ of $S$. 
Cuspidal classes never fuse in the following sense.

\begin{Theorem}[\protect{\cite[Theorem 3.2.11]{GePf2000}}] 
  \label{3.2.11}
  Let $J \subseteq  S$ and let $w \in W_J$ be  such that the conjugacy
  class $C_J$ of $w$ in $W_J$ is cuspidal in $W_J$.  Then
\[
C_J = C \cap W_J,
\]
where $C$ is the conjugacy class of $w$ in $W$.
\end{Theorem}

If $w$ is of minimal length in its conjugacy class, then it is also of
minimal length in its conjugacy class in the Coxeter group $W_{J(w)}$,
which by \cite[Proposition 3.2.12]{GePf2000} is a cuspidal class of
$W_{J(w)}$

Below, we review some basic facts about Coxeter groups of classical
type, that is of type $A$, $B$ or $D$.  For a more detailed review of
the combinatorics of the conjugacy classes of finite Coxeter groups of
classical type we refer the reader to the
description~\cite{Pfeiffer1994} of the implementation of the character
tables of these groups in GAP.

\subsection{Type \texorpdfstring{$A$}{A}.}  
\label{sec:intro-a}

Suppose $W$ is a Coxeter group of type
$A_{n-1}$.  Then $W$ is isomorphic to the symmetric group $\Sym_n$ on
the $n$ points $[n] = \{1, \dots, n\}$, with Coxeter generators $s_i =
(i, i+1)$, $i = 1, \dots, n-1$.  The \emph{cycle type} of a
permutation $w \in W$ is the partition of $n$, which contains a part
$l$ for each $l$-cycle of $w$, where fixed points count as $1$-cycles.
Since any two elements of $w$ are conjugate in $W$ if and only if they
have the same cycle type, the conjugacy classes of elements of $W$ are
naturally parametrized by the partitions of $n$.  

Here, it will be convenient to write partitions as weakly \emph{increasing}
sequences.  Given a partition $\lambda =
(\lambda_1, \dots, \lambda_t)$ of $n$ (that is a sequence of positive
integers $\lambda_1 \leq \dots \leq \lambda_t$ with $\lambda_1 + \dots
+ \lambda_t = n$), there is a corresponding parabolic subgroup $W_J =
\Sym_{\lambda_1} \times \dots \times \Sym_{\lambda_t}$ containing an
element $w$ with cycle type $\lambda$.  A particular element of
minimal length in this conjugacy class is 
the product $w_{\lambda}$ of $t$ disjoint cycles
consisting of $\lambda_i$ successive points, for $i = 1, \dots, t$.
For example, a minimal length representative of the conjugacy class 
of elements with cycle structure
$1124$  in $\Sym_8$ is $w_{1124} = (1)(2)(3,4)(5,6,7,8) = (3,4)(5,6,7,8)$.

Note that $w_{\lambda}$ is a Coxeter element of $W_J$,
the product
\begin{align} \label{eq:wa}
  w_{\lambda} = \prod_{s_i \in J} s_i
\end{align}
(in \emph{decreasing} order) of all $s_i \in J$. For example,
$J(w_{1124}) = \{s_3, s_5, s_6, s_7\}$, and $w_{1124} = s_7 s_6 s_5
s_3$.

\subsection{Type \texorpdfstring{$B$}{B}.} 
\label{sec:intro-b}

Suppose $W$ is a Coxeter group of type $B_{n}$.
Then $W$ is isomorphic to the group of permutations on $\{-n,\dots,-1,0,1, \dots, n\}$
satisfying $w(-i) = - w(i)$. Alternatively, we can represent this group
as the group of signed permutations, i.e. injective maps from $[n]$ to 
$[n] \cup -[n]$ with precisely one of $i$ and $-i$ in the image. Since the elements
are permutations, we can write them in cyclic form. We have two types of cycles: 
cycles which do not contain $i$ and $-i$ for any $i$, and cycles in which $i$ is
an element if and only if $-i$ is an element. Cycles of the first type come in 
natural pairs, and instead of $(i_1,i_2,\dots,i_k)(-i_1,-i_2,\dots,-i_k)$, we write
$(i_1,i_2,\dots,i_k)$ and call it a positive cycle. Cycles of the second type are
of the form $(i_1, i_2 \dots,i_k,-i_1,-i_2,\dots,-i_k)$. We shorten that to 
$(i_1,i_2,\ldots,i_k)^-$ and call it a negative cycle. For example, the permutation
$$-4 \mapsto -2, \: -3 \mapsto 1, \: -2 \mapsto 4, \: -1 \mapsto 3, \: 0 \mapsto 0, \: 1 \mapsto -3, \: 2 \mapsto -4, \: 3 \mapsto -1, \: 4 \mapsto 2$$
is written as $(1,-3)(2,-4)^-$. In this notation, every
signed permutation looks like an ordinary permutation in cyclic form, except that every element
and every cycle can have a minus sign. Note that we can change the sign of all elements
in a cycle without changing the signed permutation. The Coxeter generators are 
$t_1 = (1)^-$ and $s_i = (i,i+1)$, $i = 1, \dots, n-1$.
We also set $t_i = (i)^-$, for $i > 1$.

An element $w \in W(B_n)$ can also be represented in the form of a
signed permutation matrix.  This is an $n \times n$ matrix which acts
on the standard basis $\{e_1, \dots, e_n\}$ of $\R^n$ in the same way
as the permutation $w$ acts on the points $[n] = \{1, \dots, n\}$, i.e., for $i \in [n]$, it
maps $e_i$ to $e_{\Size{w(i)}}$ or its negative, depending on whether
  $w(i)$ is positive or negative.  We will briefly use this matrix
  representation of $W(B_n)$ in Section~\ref{sec:n-c}.

Since conjugation on a signed permutation in cyclic form works in the same way as
with usual permutations (if we conjugate with $w$, an element $i$ of any cycle is 
replaced by $w(i)$), two signed permutations are conjugate if and only if
they have the same number of negative cycles of every length, and the
same number of positive cycles of every length. The \emph{cycle type} of a permutation $w \in W$ is a double partition 
$\lambda = (\lambda^+,\lambda^-)$ with $\Size{\lambda^+} + \Size{\lambda^-} = n$,
so that $\lambda^+$ contains a part $l$ for each
positive $l$-cycle of $w$, and $\lambda^-$ contains a part $l$ for each
negative $l$-cycle of $w$. Two elements of $W$ are conjugate in $W$ if and only if they have the
same cycle type, and therefore the conjugacy classes of elements of $W$ are
naturally parametrized by the double partitions of $n$. For example, the conjugacy
class of $(1,5,-2)(4,7)(3)^-(6,-8)^-$ is $(21,32)$. 

Take $J \subseteq \{t_1,s_1,s_2,\ldots,s_{n-1}\}$ and $w \in W_J$. If $s_i \notin J$, the elements 
$i+1,\ldots,n$ appear in positive cycles of $w$ with all positive elements. Therefore, if we are given a double partition
$(\lambda^+,\lambda^-)$ of $n$, $\lambda^+ = (\lambda_1^+, \dots, \lambda_t^+)$,
$\lambda^- = (\lambda_1^-, \dots, \lambda_s^-)$, the smallest parabolic 
subgroup $W_J$ that contains an element of cycle type $(\lambda^+,\lambda^-)$ is of the form $W(B_{\Size{\lambda^-}}) \times \Sym_{\lambda_1^+} 
\times \dots \times \Sym_{\lambda_t^+}$. 
According to the description \cite[3.4.2]{GePf2000}
of conjugacy classes of $W$,
there is a minimal length representative
$w_{\lambda}$
of the corresponding conjugacy class of the following form. The negative cycles
contain $1,\dots,\Size{\lambda^-}$, and the positive cycles contain $\Size{\lambda^-}+1,\dots,n$;
furthermore, each cycle contains only consecutive numbers in increasing order. For example, a minimal length
representative of the conjugacy class corresponding to $\lambda = (112,23)$ is 
$w_{\lambda} = (1,2)^-(3,4,5)^-(6)(7)(8,9)$.

\subsection{Type \texorpdfstring{$D$}{D}.} 
\label{sec:intro-d}

Suppose $W$ is a Coxeter group of type $D_{n}$.  Then $W$ is
isomorphic to the group of permutations on $\{-n,\dots,-1,0,1, \dots,
n\}$ satisfying $w(-i) = - w(i)$ and with an even number of $i>0$
satisfying $w(i)<0$. Alternatively, we can represent this group as the
group of signed permutations with an even number of $i$ mapping to
$-[n]$. These are precisely the signed permutations with an even number
of negative cycles. The Coxeter generators are $u = (1, -2)$ and $s_i
= (i,i+1)$, $i = 1, \dots, n-1$.  The \emph{cycle type} of a
permutation $w \in W$ is a double partition $(\lambda^+,\lambda^-)$
with $\Size{\lambda^+} + \Size{\lambda^-} = n$, so that $\lambda^+$ contains a
part $l$ for each positive $l$-cycle of $w$, and $\lambda^-$ contains
a part $l$ for each negative $l$-cycle of $w$. If two elements of $W$
are conjugate, they have the same cycle type. Having the same cycle type,
however, is not a sufficient condition for conjugacy. For example, $u$ and $s_1$
have the same cycle type but are not conjugate. It is easy to see that if
they have the same cycle type $(\lambda^+,\lambda^-)$ and $\Size{\lambda^-} > 0$ or
$\lambda^+$ contains an odd part, they are conjugate. If they have the
same cycle type $(\lambda^+,\emptyset)$, where $\lambda^+$ contains
only even parts, they are conjugate if and only if the number of
negative numbers in their cycle decomposition has the same parity.

We call a partition \emph{even} if it consists of even parts only.
The conjugacy classes of elements of $W$ are naturally parametrized by
double partitions of $n$, where $\lambda^-$ has an even number of
parts, with two classes when $\lambda^- = \emptyset$ and $\lambda^+$
is even.  Given a double partition $(\lambda^+,\lambda^-)$ of $n$,
$\lambda^+ = (\lambda_1^+, \dots, \lambda_t^+)$, $\lambda^- =
(\lambda_1^-, \dots, \lambda_s^-)$, $s$ even, the corresponding
parabolic subgroup $W_J$ is of the form $W(D_{\Size{\lambda^-}}) \times
\Sym_{\lambda_1^+} \times \dots \times \Sym_{\lambda_t^+}$.
Furthermore, there is a minimal length representative $w_{\lambda}$ of the
corresponding conjugacy class of the following form. The negative
cycles contain $1,\dots,\Size{\lambda^-}$, the positive cycles contain
$\Size{\lambda^-}+1,\dots,n$, and each cycle contains only
consecutive numbers in increasing order. If $\lambda^- = \emptyset$
and $\lambda^+$ has only even parts, then there is an extra representative
$w_{\lambda}'$ with the first positive cycle starting with $-1$ instead of $1$.  For
example, for $\lambda = (112,23)$, we have $w_{\lambda} = (1,2)^-(3,4,5)^-(6)(7)(8,9)$, and
for $(224,\emptyset)$, we have $w_{\lambda} = (1,2)(3,4)(5,6,7,8)$ and
$w_{\lambda}' = (-1,2)(3,4)(5,6,7,8)$.

\section{Centralizers.}
\label{sec:centralizers}

In this section we prove a general theorem about the structure of
element centralizers in finite Coxeter groups.  It is shown to be a
consequence of Theorem~\ref{3.2.11}, which in the book~\cite{GePf2000}
 has
been established by a careful case-by-case analysis.  Without loss of
generality, we may assume that $w \in W$ is an element of minimal
length in its conjugacy class in $W$.

\begin{Theorem} \label{thm:2} 
  Suppose $w \in W$ has minimal length in its
  conjugacy class in $W$ and let $J = J(w)$. Then $C_W(w)W_J = N_W(W_J)$.
\end{Theorem}

Clearly, part (i) of Theorem~\ref{thm:1} follows from this result.

\begin{proof}
Denote by $C$ the conjugacy class of $w$ in $W$ and by $C_J$ its
conjugacy class in $W_J$, which is cuspidal in $W_J$.  
By Theorem~\ref{3.2.11}, $C_J = C \cap W_J$, which implies 
that for every $x\in N_W(W_J)$ there exists an element $u \in W_J$ with
$w^{x} = w^u$. So $xu^{-1}\in C_W(w)$, i.e. $x\in C_W(w)W_J$, and hence
$N_W(W_J)\subseteq C_W(w)W_J$.

Now it only remains to show that $C_W(w)\subseteq N_W(W_J)$.  Let $y
\in C_W(w)$ and write it as $y = u x v^{-1}$ for $u, v \in W_J$ and a
double coset representative $x \in X_{JJ}$.  From $w \in C_J \cap
C_J^y$ it then follows that $w^v \in C_J \cap C_J^x \subseteq W_J \cap
W_J^x = W_{J \cap J^x}$, by Theorem~\ref{2.1.12}, and since
$C_J$ is a cuspidal class in $W_J$, we must have $J \cap J^x = J$,
whence $x \in N_J \subseteq N_W(W_J)$ and thus $y = u x v^{-1} \in
N_W(W_J)$.
\end{proof}

The following additional results are of independent interest
and will be used in the proof of Theorem~\ref{thm:3} below.

\begin{Proposition} \label{prop:1}
  Suppose $w \in W$ has minimal length in its conjugacy class in $W$ 
  and let $J = J(w)$.  Then the following hold.
\begin{enumerate}
\item The conjugacy class of $w$ in $W$ is a disjoint union of
  conjugates of the conjugacy class of $w$ in $W_J$.
\item If $a \in C_W(w)$ and $x \in X_J$ are such that $C_{W_J}(w) a
  \subseteq W_J x$, then $x \in N_J$.
\item If $x \in X_J$ is such that $\ell(w^x) = \ell(w)$, then $J(w^x) =
  J^x$.
\item If $v \in W$ is such that $\ell(w^v) = \ell(w)$, then $J(w^x) =
  J^x$, where $v = u \cdot x$ with $u \in W_J$ and $x \in X_J$.
\end{enumerate}
\end{Proposition}
\begin{proof}
  (i) and (ii) follow from the proof of Theorem~\ref{thm:2}.

  (iii) Let $K = J(w^x)$.  By Proposition~\ref{3.1.11}, there exists
  an element $y \in X_{KJ}$ such that $K^y = J$.  Hence $w^{xy} \in
  W_J$ and $\ell(w^{xy}) = \ell(w^x) = \ell(w)$ and (since $C \cap W_J
  = C_J$)  $w^{xy} = w^u$ for some $u \in W_J$.  Moreover, $X_K =
  y X_J$. Hence $u^{-1}xy$ centralizes $w^u$, and if we write $u^{-1}xy
  = a \cdot z$ for $a \in W_J$ and $z \in X_J$ then, by (ii), $z \in
  N_J$.  It follows that $z y^{-1} \in X_{JK}$ is the unique
  minimal length representative of the coset $W_Jx$, hence $x = z
  y^{-1}$ and $J^x = J(w^x)$.
  
  (iv) We have $w^v = (w^u)^x$. Conjugation with $x$ does not decrease the length
  (Lemma~\ref{2.1.14}), so $\ell(w) = \ell(w^v) \geq \ell(w^u)$ and 
  therefore $\ell(w^u) = \ell(w)$. By (iii), with $w$ replaced by $w^u$, we have
  $J(w^v) = J(w^u)^x$, and $J(w^u) = J(w)$, which finishes the proof.
\end{proof}

\section{Complements.}
\label{sec:cases}

In this section we prove part (ii) of Theorem~\ref{thm:1} for each type of 
irreducible finite Coxeter group, case by case.
We start with the general observation that part (ii) of the theorem is
straightforward in the following situations.

\begin{Lemma}
  \label{la:cori} 
  Let $w \in W$ be an element of minimal length in its conjugacy class
  in $W$, and let $J = J(w)$.  If $w$ is cuspidal in $W$ or if
  $C_{W_J}(w) = W_J$ then $N_J$ is a complement of $C_{W_J}(w)$ in $C_W(w)$.
\end{Lemma}

\begin{proof}
  If $w$ is cuspidal then $W_J = W$ and both quotients $N_W(W_J)/W_J$
  and $C_W(w)/C_{W_J}(w)$ are trivial.

  If $C_{W_J}(w) = W_J$ then 
$w = w_J$ and
$C_W(w) = N_W(W_J)$~\cite[Proposition 2.2]{PfeifferRoehrle2005}.
\end{proof}

Our general strategy in search of a centralizer complement for $w$ will
be to identify a complement $M$ of $W_J$ in its normalizer that
centralizes $w$.  More precisely, we have the following consequence
of Theorem~\ref{thm:2}.

\begin{Proposition}
  \label{pro:M}
  Let $w \in W$ be of minimal length in its conjugacy class, let $J =
  J(w)$ and suppose that the normalizer complement $N_J$ is generated
  by elements $x_1, \dots, x_r$.  Let $u_1, \dots, u_r \in W_J$ be
  such that $u_i x_i \in C_W(w)$, $i = 1, \dots, r$, and set $M =
  \langle u_1 x_1, \dots, u_r x_r\rangle$.  Then $M$ is a complement
  of $C_{W_J}(w)$ in $C_W(w)$ provided that $M \cap W_J = 1$.
\end{Proposition}

\begin{proof}
  Clearly, $W_J M = W_J N_J = N_W(W_J)$.  From $M \cap W_J = 1$ it
  then follows that $M$ is a complement of $W_J$ in its normalizer.
  Moreover, $M$ is a subgroup of $C_W(w)$ since each of its generators
  centralizes~$w$.  From Theorem~\ref{thm:2} it then follows that $M$
  is a complement of $C_{W_J}(w)$ in $C_W(w)$.
\end{proof}

\subsection{Type \texorpdfstring{$A$}{A}.} 
\label{sec:a}

Let $\lambda = (1^{a_1}, 2^{a_2}, \dots, n^{a_n}) \partition n$,
let $w_{\lambda}$ be as in (\ref{eq:wa}) and let $J = J(w_{\lambda})$.
Then $W_J$ is a direct product
\begin{align*}
  W_J = \Sym_1^{a_1} \times \Sym_2^{a_2} \times \dots \times \Sym_n^{a_n}
\end{align*}
of symmetric groups, and its normalizer 
\begin{align*}
  N_W(W_J) = \Sym_1 \wr \Sym_{a_1} \times 
\Sym_2 \wr \Sym_{a_2} \times 
\dots \times \Sym_n \wr \Sym_{a_n}
\end{align*}
is a direct product of wreath products of symmetric with symmetric groups.  In a similar way,
the centralizer
\begin{align*}
  C_W(w_{\lambda}) = C_1 \wr \Sym_{a_1} \times C_2 \wr \Sym_{a_2} \times \dots \times C_n \wr \Sym_{a_n}
\end{align*}
is a direct product of wreath products of cyclic with symmetric groups,
and the centralizer
\begin{align*}
  C_{W_J}(w_{\lambda}) = C_1^{a_1} \times C_2^{a_2} \times \dots \times C_n^{a_n}.
\end{align*}
is a direct product of cyclic groups.  Clearly, the quotients
$N_W(W_J)/W_J$ and $C_W(w_{\lambda})/C_{W_J}(w_{\lambda})$ are both
isomorphic to $\Sym_{a_1} \times \Sym_{a_2} \times \dots \times
\Sym_{a_n}$.  In order to show that the particular normalizer
complement $N_J$ is also a complement of $C_{W_J}(w_{\lambda})$ in
$C_W(w_{\lambda})$, we introduce some notation.  Let us define as
\begin{align} \label{eq:sss}
s(o,m)
= (s_{o+1} s_{o+2} \dotsm s_{o+2m-1})^m
= (o+1,o+2, \dots, o+2m)^m
\end{align}
the permutation that swaps, after an offset $o$, two adjacent blocks
of $m$ points $\{o+1, \dots, o+m\}$ and $\{o+m+1, \dots, o+2m\}$.
For example, $s(2,3) = (s_3 s_4 s_5 s_6 s_7)^3 = (3, 4, 5, 6, 7, 8)^3 = (3,6)(4,7)(5,8)$.  And $s_i = s(i-1, 1)$.
Then 
\begin{align*}
  N_J = \Sym_{a_1} \times \Sym_{a_2} \times \dots  \times \Sym_{a_n}
\end{align*}
is a direct product of symmetric groups
$\Sym_{a_m}$, with Coxeter generators 
$s(o_m, m)$,    
$s(o_m{+}m,m)$,    \dots,
$s(o_m{+}(a_m{-}2)m, m)$,
and offsets 
\begin{align} \label{eq:om-a}
  o_m = a_1 + 2 a_2 + \dots + (m-1) a_{m-1},
\end{align}
for those $m \in \{ 1, \dots, n\}$ with $a_m > 0$.  

\begin{Proposition} 
  \label{pro:a} 
  Let $\lambda \partition n$, let $w_{\lambda}$ be the permutation
  with cycle structure $\lambda$ from (\ref{eq:wa})  and let $J =
  J(w_{\lambda})$ be the corresponding subset of $S$.
Then $N_J$ is a complement
  of $C_{W_J}(w_{\lambda})$ in $C_W(w_{\lambda})$.
\end{Proposition}

\begin{proof}
  It suffices to consider the case $\lambda = (m^a)$ since all of
  $W_J$, $N_W(W_J)$, $C_W(w_{\lambda})$, $C_{W_J}(w_{\lambda})$, and
  $N_J$ are subgroups of the direct product
  \begin{align*}
    \Sym_{1a_1} \times \Sym_{2a_2} \times \dots \times \Sym_{n a_n}
  \end{align*}
  inside $\Sym_n$ and one can argue componentwise.

  If $\lambda = (m^a)$, then $N_J$ is isomorphic to $\Sym_a$, with
  $a-1$ Coxeter generators $s(o,m)$, $s(m,m)$, \dots, $s((a{-}2)m,m)$,
  permuting the blocks of $m$ points
  \begin{center}
    $\{1, \dots, m\}$, $\{m{+}1, \dots, 2m\}$, \ldots,
    $\{(a{-}1)m{+}1,\dots, am\}$.
  \end{center}
  Clearly 
  \begin{align*}
    w_{\lambda} = (1, \dots, m)(m{+}1, \dots, 2m) \dotsm
    ((a{-}1)m{+}1,\dots, am)
  \end{align*}
  is centralized by $N_J$.
  The claim now follows from Theorem~\ref{thm:2}.
\end{proof}

\subsection{Type \texorpdfstring{$B$}{B}.} 
\label{sec:b}

Let $\lambda \dpartition n$ with $\lambda^+ = (1^{a_1}, 2^{a_2}, \dots, n^{a_n})$ and $\lambda^- = (1^{b_1}, 2^{b_2}, \dots, n^{b_n})$,
let $w_{\lambda}$ be as in Section~\ref{sec:intro-b},
and let $J = J(w_{\lambda})$ be the corresponding subset of $S$.
Then $W_J$ is a direct product
\begin{align*}
  W_J = W(B_{\Size{\lambda^-}}) \times \Sym_1^{a_1} \times \Sym_2^{a_2} \times \dots \times \Sym_n^{a_n}
\end{align*}
and its normalizer
\begin{align*}
  N_W(W_J) = W(B_{\Size{\lambda^-}}) \times \Sym_1 \wr W(B_{a_1})
\times \Sym_2 \wr W(B_{a_2}) \times \dots
\times \Sym_n \wr W(B_{a_n})
\end{align*}
is a direct product of $W(B_{\Size{\lambda^-}})$
with wreath products of symmetric groups and
Coxeter groups of type $B$.
In a similar way, the centralizer
\begin{align*}
  C_W(w_{\lambda}) = 
C_{W(B_{\Size{\lambda^-}})}(w_{\lambda})
\times C_1 \wr W(B_{a_1}) \times C_2 \wr W(B_{a_2}) \times \dots \times C_n \wr
  W(B_{a_n})
\end{align*}
is a direct product of $C_{W(B_{\Size{\lambda^-}})}(w_{\lambda})$ and  wreath products, and the centralizer
\begin{align*}
C_{W_J}(w_{\lambda}) = 
  C_{W(B_{\Size{\lambda^-}})}(w_{\lambda}) \times 
  C_1^{a_1} \times C_2^{a_2} \times \dots \times C_n^{a_n}
\end{align*}
is a direct product of $C_{W(B_{\Size{\lambda^-}})}(w_{\lambda})$
and cyclic groups.
Clearly, the quotients
$N_W(W_J)/W_J$ and $C_W(w_{\lambda})/C_{W_J}(w_{\lambda})$ are both
isomorphic to $W(B_{a_1}) \times W(B_{a_2}) \times \dots \times
W(B_{a_n})$.  In order to show that a variant of the particular normalizer
complement $N_J$ is a complement of $C_{W_J}(w_{\lambda})$ in
$C_W(w_{\lambda})$, we introduce some more notation.

Denote by $r(o,m)$ the  permutation
defined by
\begin{align*}
  x.r(o,m) =
  \begin{cases}
    2o+m+1-x, & \text{if } o+1 \leq x \leq o+m, \\
    x, & \text{otherwise.}
  \end{cases}
\end{align*}
In this way, $r(o,m) = (o+1,o+m)(o{+}2, o{+}m{-}1) \dotsm $ reverses the range
$\{o+1, \dots, o+m\}$ and thus is the longest element of the
symmetric group $\Sym_{\{o+1, \dots, o+m\}}$
with Coxeter generators $s_{o+1}, \dots, s_{o+m-1}$.
For example, $r(2,5) = (3,7)(4,6)$.

Moreover, denote 
\begin{align*}
  t(o, m) = (o{+}1)^- (o{+}2)^- \dotsm (o{+}m)^-,
\end{align*}
which acts as $-1$ on the
points $\{o+1, o+2, \dots, o+m\}$ and as identity everywhere else.

If $\lambda^+ = (m^a)$ and $\lambda^- = \emptyset$, then $W_J$ is a
direct product of $a$ copies of $\Sym_m$ and $N_J$ is isomorphic to
$W(B_a)$, with Coxeter generators 
\begin{center}
  $r(0,m)\, t(0,m)$ and $s(0,m)$, $s(m,m)$, \dots, $s((a{-}2)m,m)$.
\end{center}
In general, if $\lambda^+ = (1^{a_1}, 2^{a_2}, \dots, n^{a_n})$, then
$W_J$ is a direct product of $W(B_{\Size{\lambda^-}})$ and direct products
of isomorphic symmetric groups $\Sym_m$ and $N_J$ is a direct product
of groups $W(B_{a_m})$, with Coxeter generators 
\begin{center}
  $r(o_m,m)\, t(o_m, m)$ and $s(o_m,m)$,
$s(o_m + m,m)$, \dots, $s(o_m + (a_m{-}2)m,m)$ 
\end{center}
and offsets 
\begin{align} \label{eq:om-b}
  o_m = \Size{\lambda^-} + a_1 + 2 a_2 + \dots + (m-1)a_{m-1},
\end{align}
for those $m \in \{ 1, \dots, n\}$ with $a_m > 0$.  Unfortunately,
this group $N_J$ usually does not centralize $w_{\lambda}$.  However,
if we define a group $N_{\lambda}$ as the subgroup of $W$ generated by
the same elements as $N_J$, with $t(o_m,m)$ in place of $r(o_m,m)\,
t(o_m,m)$, then $N_{\lambda}$ is a centralizing complement.

\begin{Proposition} 
  \label{pro:b}
  Let $\lambda \dpartition n$, let $w_{\lambda}$ be
  as in Section~\ref{sec:intro-b}, and let $J = J(w_{\lambda})$ be the corresponding subset
  of $S$.  Then
\begin{align*}
  N_{\lambda} =
\Span{
t(o_m, m), \,
s(o_m{+}km, m)
\mid k=0,\dots,a_m{-}2,\,
 m = 1, \dots, n,\,
a_m > 0}
\end{align*}
is a complement of
  $C_{W_J}(w_{\lambda})$ in $C_W(w_{\lambda})$.
\end{Proposition}

\begin{proof}
  Clearly, $N_{\lambda}$ centralizes $w_{\lambda}$ since its
  generators $t(o_m,m)$ and $s(o_m{+}km, m)$ do.  The statement now
  follows with Proposition~\ref{pro:M} from the fact that
  $N_{\lambda}$ is a complement of $W_J$ in its normalizer in~$W$.
\end{proof}

\subsection{Type \texorpdfstring{$D$}{D}.}  
\label{sec:d}

The case of Coxeter groups of type $D_n$ is best dealt with by
comparing it to the situation in type $B_n$.  Throughout this section,
we assume $n \geq 4$, denote by $W$ the Coxeter group of type $B_n$
with Coxeter generators $S$, as described in
Section~\ref{sec:intro-b}, and by $W^+$ the Coxeter group of type
$D_n$ with Coxeter generators $S^+$, consisting of the signed
permutations with an even number of negative cycles as described in
Section~\ref{sec:intro-d}.

The following properties are easy to establish and we leave
their proofs to the reader.

\begin{Lemma}
\label{la:d}
  Let $w \in W$ be an element of cycle type $\lambda = (\lambda^+, \lambda^-)$
such that $\lambda^-$ has an even number of parts and that
$w$ has minimal length in its conjugacy class in $W$. 
Also let $J = J(w)$.  Then the following hold.
\begin{enumerate}
\item $w$ has minimal length in its class in $W^+$.
\item If $\lambda^- = \emptyset$ and $\lambda^+$ is even then
  $C_{W^+}(w) = C_W(w)$, otherwise $C_{W^+}(w)$ has index $2$ in
  $C_W(w)$.
\item If $J^+ = S^+ \cap W_J$ then $W^+_{J^+}$ is the smallest
  parabolic subgroup of $W^+$ containing $w$.
\item 
If $\lambda^- = \emptyset$ then $J^+ = J$, otherwise
$W^+_{J^+} = W_J \cap W^+$ is a subgroup of index $2$
in $W_J$ 
\item $C_{W^+_{J^+}}(w) = C_{W_J}(w) \cap W^+$ is a subgroup of index
$2$ in  $C_{W_J}(w)$ unless $\lambda^- = \emptyset$.
\item If $\lambda^+$ is not even and $\lambda^- = \emptyset$ then
  $N^+_{J^+} = N_J \cap W^+$ is a subgroup of index $2$ in $N_J$,
  otherwise $N_J \cong N^+_{J^+}$.
\end{enumerate}
\end{Lemma}

The  parabolic subgroup $W^+_{J^+}$ is of the form 
$D_{\Size{\lambda^-}} \times \Sym_{\lambda_1^+} \times \dots \times \Sym_{\lambda_t^+}$,
where $D_m$ is the subgroup of $W^+$ generated by $\{u, s_1, \dots, s_{m-1}\}$,
for $m =  2, \dots, n$.

\begin{Definition} 
  \label{def:non-compliant} 
  We call a double partition $\lambda = (\lambda^+, \lambda^-)$ a
  \emph{non-compliant double partition} if  $\lambda^+$
  consists of a single odd part $m$ and $\lambda^-$ is a nonempty even
  partition of even length.

  We call a conjugacy class $C$ of $W$ a \emph{non-compliant class},
  if, for some odd $n > 4$, there is a non-compliant double partition
  $\lambda$ of $n$ and a parabolic subgroup $W_M$ of $W$ which has an
  irreducible component $W_K$ of type $D_n$, such that $C$ contains an
  element of $W_M$ whose projection on $W_K$ has cycle type $\lambda$.
\end{Definition}

For example, the elements of $W = W(D_5)$ with cycle type $(1, 22)$
form a non-compliant class.  For another example, the elements of $W =
W(D_7)$ of cycle type $(21, 22)$ form a non-compliant class, as some
of them lie in a parabolic subgroup $W_M$ of type $D_5 \times A_1$,
with $D_5$-part of cycle type $(1,22)$.

\begin{Lemma}
  An element $w \in W(D_n)$ of cycle type $\lambda = (\lambda^+,
  \lambda^-)$ lies in a non-compliant class if and only if $\lambda^+$
  is not even and $\lambda^-$ is nonempty and even.
\end{Lemma}

The next result shows that, in a Coxeter group $W^+$ of type $D_n$,
the centralizer $C_{W^+}(w)$ splits over $C_{W^+_J}(w)$, unless the
class of $w \in W^+$ is non-compliant.  Here, we write $J^+(w)
\subseteq S^+$ for the set of generators occurring in a reduced
expression of $w$ when considered as an element of $W^+$, in order to
distinguish it from the set $J(w) \subseteq S$
of generators in a reduced expression of $w \in W$.

\begin{Proposition} 
  \label{pro:d} 
  Let $\lambda = (\lambda^+,\lambda^-) \dpartition n$ be such that
  $\ell(\lambda^-)$ is even.
Let $w_{\lambda}$  and $N_{\lambda}$ be as in Proposition~\ref{pro:b}
  and let $J^+ = J^+(w_{\lambda})$ be the corresponding subset of $S^+$.
  Then the following hold.
  \begin{enumerate}
  \item If $\lambda^+$ is even then $N_{\lambda}$ is a complement of
  $C_{W^+_{J^+}}(w_{\lambda})$ in $C_{W^+}(w_{\lambda})$.
\item If $\lambda^+$ is not even and $\lambda^- = \emptyset$ then
  $N_{\lambda} \cap W^+$ is a subgroup of index $2$ in $N_{\lambda}$
  and a complement of $C_{W^+_{J^+}}(w_{\lambda})$ in
  $C_{W^+}(w_{\lambda})$.
\item If $\lambda^+ = (1^{a_1}, \dots, n^{a_n})$ and $\lambda^- =
  (\lambda^-_1, \dots, \lambda^-_s)$ is not even then there is an
  index $j \leq s$ such that $k = \lambda^-_1 + \dots + \lambda^-_j$
  is odd, and the subgroup
\begin{align*}
  N^+_{\lambda} =
\Span{
t(0, k)^m\, t(o_m, m), \,
s(o_m{+}im, m)
\mid i=0,\dots,a_m{-}2,\,
 m = 1, \dots, n,\,
a_m > 0}
\end{align*}
is a complement of
  $C_{W^+_{J^+}}(w_{\lambda})$ in $C_{W^+}(w_{\lambda})$.
  \end{enumerate}
\end{Proposition}

Note that $t(0,k)^m = 1$ if $m$ is even
and $t(0,k)^m = t(0,k)$ if $m$ is odd.

\begin{proof}
Let $J = J(w_{\lambda})$ be the subset of $S$ corresponding to $\lambda$.
In all three cases it suffices to find a complement $N^*$
of $W_J$ in its normalizer in $W$ that centralizes $w_{\lambda}$ 
such that $\Size{N^* \cap W^+} = \Size{N^+_{J^+}}$.
For then $N^* \cap W^+_{J^+} = 1$ and
$N_{W^+}(W^+_{J^+}) \subseteq N_W(W_J) = W_J N^*$
imply that $N^* \cap W^+$ is a 
complement of 
$W^+_{J^+}$ in its normalizer in $W^+$ that
centralizes $w_{\lambda}$, and the claim follows with
Proposition~\ref{pro:M}.

(i)
If $\lambda^+$ is even then 
$N_{\lambda}$ is contained in $W^+$  and $N^* = N_{\lambda}$ will do.

(ii)
If $\lambda^- = \emptyset$ and $\lambda^+$ is not even
then $J^+ = J$ but $N_{J^+}$ is subgroup of index $2$ in $N_J$
and $N^* = N_{\lambda} \cap W^+$ will do.

(iii) If $\lambda^-$ is not even
then $N^+_{\lambda}$ is a 
complement of $W_J$ in its normalizer in 
$W$ that is contained in  $W^+$
and centralizes $w_{\lambda}$,
whence $N^* = N^+_{\lambda}$ will do.
\end{proof}

\subsection{Type \texorpdfstring{$I$}{I}.} 
\label{sec:i}

Suppose $W$ is a Coxeter group of type $I_2(m)$.  Then $W$ is the
group generated by generators $s_1$ and $s_2$ satisfying $(s_1s_2)^m =
(s_2s_1)^m$. Each element of $W$ is either cuspidal or an
involution. Hence the theorem for this type follows from
Lemma~\ref{la:cori}.

\subsection{Exceptional Types.}  
\label{sec:exceptional}

Although in type $A$ each conjugacy class contains an element $w$ such
that the normalizer complement $N_J$ is also a centralizer complement,
this cannot be expected in general to be the case.  However, from the
preceding examples one sees that it is frequently possible to
construct from $N_J$ an isomorphic copy $N_J^*$ which is a centralizer
complement.  In each of the above examples, $N_J^*$ is obtained from
$N_J$ by replacing generators $x_i$ of $N_J$ by products $w_L x_i$ for
suitable subsets $L \subseteq J$.

Based on this observation,
we formulate an algorithm, which in practice  always
finds a centralizer complement, except for elements of non-compliant
classes.

\bigskip
\noindent
\textbf{Algorithm} CentralizerComplement.
\nopagebreak

\smallskip
\noindent
\textbf{Input:} A finite Coxeter group $W$ and an element $w$
of minimal length in its conjugacy class in $W$.

\smallskip
\noindent
\textbf{Output:} a centralizer complement for $w$, or fail if none exists.
\begin{enumerate} \renewcommand{\labelenumi}{\arabic{enumi}.}
\item set $J \gets J(w)$.
\item find involutions $x_1, \dots, x_r$ generating the normalizer
complement $N_J$.
\item for each element $v$ of minimal length in the $W_J$-conjugacy
  class of $w$ do the following:
\begin{itemize}
\item let $u \in W_J$ be such that $v^u = w$;
\item  for each $i = 1, \dots, r$, set 
\begin{align*}
  Y_i \gets 
  \begin{cases}
    \{x_i\}, & \text{if } v^{x_i} = v, \\
\{w_L x_i: L \subseteq J,\, x_i^{w_L} = x_i,\, v^{w_L} = v^{x_i}\}, & \text{otherwise.}
  \end{cases}
\end{align*}
\item if there are elements $y_i \in Y_i$, $i = 1, \dots, r$, such that
$M = \langle y_1, \dots, y_r\rangle$ satisfies
 $M \cap W_J = 1$
then return $M^u$.
\end{itemize}
\item return fail (if we ever get here).
\end{enumerate}

Note that, by Proposition~\ref{pro:M}, any group $M$ found in this way
is necessarily a complement of the centralizer of $w$ in~$W_J$.

For $W$ irreducible of exceptional type, the algorithm produces a centralizer
complement in all but seven cases.  Each case corresponds to a
non-compliant class from the following table.  In this table we list,
for each non-compliant class $C$ of $W$, its position $i$ in CHEVIE's
list of conjugacy classes of $W$, its name, a reduced expression for a
representative $w$ of minimal length, the set $J(w)$, a set $M
\supseteq J(w)$, the type of $W_M$ exhibiting a direct factor of type
$D_{2l+1}$, and the label $\lambda$ of the conjugacy class of
$W(D_{2l+1})$ containing the projection of $w$.
\begin{align*}
\begin{array}{cccccccc}
  \hline
W & i & \text{name} & w \in C & J(w) & M & \text{type} & \lambda \\
  \hline
  \hline
E_6 & 7 & D_4(a_1) & {}_{342345} & 2345 & 12345 & D_5 & (1,22) \\
  \hline
  \hline
E_7 & 9 & D_4(a_1) & {}_{425423} & 2345 & 12345 & D_5 & (1,22) \\
    & 42 & D_4(a_1)+A_1 & {}_{4254237} & 23457 & 123457  & D_5 \times A_1 & (1,22) \\
  \hline
  \hline
E_8 & 16 & D_4(a_1) & {}_{242345} & 2345 & 12345 & D_5 & (1,22) \\
    & 44 & D_6(a_1) & {}_{24234567} & 234567 & 2345678 & D_7 & (1,42) \\
    & 53 & D_4(a_1)+A_2 & {}_{34234578} & 234578 & 1234578 & D_5 \times A_2 & (1,22) \\
    & 73 & D_4(a_1)+A_1 & {}_{3542348} & 23458 & 123458 & D_5 \times A_1 & (1,22) \\
  \hline
\end{array}
\end{align*}

In the next section we show that in all of these cases, and indeed
whenever $w$ lies in a non-compliant class, no complement exists.
 
\subsection{Non-compliance.}
\label{sec:n-c}

The class of $W(D_5)$ with label $\lambda = (1, 22)$ contains the element 
\begin{align*}
  w_{\lambda} 
= ts_1 s_2s_1ts_1s_2s_3 
= (ts_1 t)s_2(ts_1t)s_1s_2s_3 
= us_2us_1s_2s_3, 
\end{align*}
which lies in the parabolic subgroup $W_J$ of
type $D_4$.  Its centralizer $C_{W_J}(w)$ in $W(D_4)$ has order $16$
and its centralizer $C_W(w)$ in $W(D_5)$ has order $32$.  However,
$C_{W_J}(w)$ has no complement in $C_W(w)$, since the coset $C_W(w)
\setminus C_{W_J}(w)$ contains no element of order $2$ (as a
straightforward computation in GAP will confirm).

The next result shows that this is indeed always the case, when the
cycle type of $w \in W(D_n)$ is a non-compliant double partition
of~$n$.

\begin{Proposition} \label{prop:non-compliant} 
  Suppose that $W$ is of type $D_n$
  and let $w \in W$ be an element of minimal length in its conjugacy
  class with $J(w) = J$.  If the cycle type of $w \in W$ is a
  non-compliant double partition of $n$ then the centralizer
  $C_{W_J}(w)$ has no complement in $C_W(w)$.
\end{Proposition}

\begin{proof}
Recall from Section~\ref{sec:intro-b} that elements of $W(B_n)$ can be represented as signed
permutation matrices, i.e., matrices with 
exactly one non-zero entry $1$ or $-1$ in each row and column.
Such an element  lies in $W(D_n)$ if and only if its matrix has
an even number of entries $-1$,
and it is an involution if and only if the matrix is symmetric.

Now suppose that $n = m + k$ is odd and that $w \in W = W(D_n)$ is an
element of minimal length in a conjugacy class with cycle type
$\lambda = (\lambda^+, \lambda^-)$ where the partition $\lambda^+$
consists of a single odd part $m$ and $\lambda^-$ is a nontrivial
partition of an even number $k$, consisting of an even number of even
parts.  Then $W_J$ for $J = J(w)$ has type $D_k \times A_{m-1}$ and,
by the description of $N_J$ in Section~\ref{sec:b} and
Lemma~\ref{la:d}(vi), its normalizer $N_W(W_J)$ has a complement of
order $2$ (and of type $B_1$), generated by the quotient $w_J w_0$.

We may assume that $J = S \setminus \{s_{k+1}\}$, so that, as signed
permutation on the set $\{1, \dots, n\}$, the element $w$ induces an
even number of negative cycles on the $k$ points $\{1, \dots, k\}$ and
a positive $m$-cycle on the $m$ points $\{k+1, \dots, n\}$.  The
centralizer $C_W(w)$ cannot move points from outside the $m$-cycle
into the $m$-cycle and thus consists of block diagonal matrices
\begin{align*}
\diag(A, B) = \left[
\begin{array}{cc}
  A & 0 \\ 0 & B
\end{array}
\right],
\end{align*}
of a $k \times k$ matrix $A$ and 
an $m \times m$-matrix $B$,
which modulo $2$ have the same number of
entries $-1$
since $C_W(w)$ is a subgroup of $W(D_n)$.
Moreover, for each element $\diag(A, B)$ in $C_W(w)$
the number of entries $-1$ on the diagonal of $A$,
is even, since with every point in $\{1, \dots, k\}$
being mapped to its negative, the entire
cycle which contains it must be negated.

The centralizer of $w$ in $W_J$ consists precisely
of those elements $\diag(A, B) \in C_W(w)$ which have an
even number of entries $-1$ in both $A$ and $B$,
since $A$ is the matrix  of an element in $W(D_k)$.

Let $u$ be an involution in $C_W(w)$.  Then its matrix $\diag(A, B)$
is symmetric, and an even number of entries $-1$ on the diagonal of
$A$ implies that both $A$ and $B$ have an even number of entries $-1$
and thus $u \in C_{W_J}(w)$.

It follows that $C_{W_J}(w)$ has no complement in $C_W(w)$.
\end{proof}

More generally, if $w$ lies in a non-compliant class of 
a finite Coxeter group $W$, then
its centralizer has no complement.

\begin{Theorem} \label{thm:non-compliant} 
  Let $W$ be a finite Coxeter
  group.  Suppose $w$ is an element of minimal length in a
  non-compliant conjugacy class of $W$ with $J(w) = J$.  Then the
  centralizer $C_{W_J}(w)$ has no complement in $C_W(w)$.
\end{Theorem}

\begin{proof}
  Suppose first that $W$ is a direct product $W_1 \times W_2$ of
  nontrivial standard parabolic subgroups $W_1$ and $W_2$, that $w =
  w_1 w_2$ with $w_1 \in W_1$ and $w_2 \in W_2$, and that $w_1$ lies
  in a non-compliant class of $W_1$.  Then $W_J = W_{J_1} \times
  W_{J_2}$ for certain subsets $J_1 \subseteq W_1 \cap S$ and $J_2
  \subseteq W_2 \cap S$.  If $C_{W_{J_1}}(w_1)$ has no complement in
  $C_{W_1}(w_1)$ then $C_{W_J}(w)$ cannot have a complement in $C_W(w)
  = C_{W_1}(w_1) \times C(w_2)$.

  Next,  suppose that
$w \in W_L$ for some
 $L \subseteq S$ and that
$w$ lies in a non-compliant
class of $W_L$.
  Suppose $N$ is a complement
of $C_{W_J}(w)$ in $C_W(w)$, that is
$C_W(w) = C_{W_J}(w) \rtimes N$.
Then the centralizer of $w$ in  $W_L$, 
\begin{align*}
  C_{W_L}(w) 
= C_W(w) \cap W_L 
= (C_{W_J}(w) \rtimes N) \cap W_L
= C_{W_J}(w) \rtimes (N \cap W_L),
\end{align*}
in contradiction to our assumption that $w$ lies in a
non-compliant class.

The theorem now follows from Definition~\ref{def:non-compliant} and
Proposition~\ref{prop:non-compliant}.
\end{proof}

\section{Applications.}
\label{sec:application}

In this section we first use Theorem \ref{thm:1} to prove a result about
minimal length representatives of conjugacy classes. Then we show how it
implies the celebrated Solomon's character formula. Finally, we
discuss the interpretation of Solomon's theorem as a Coxeter group
analogue of MacMahon master theorem.

\begin{Theorem} 
  \label{thm:3} 
  Assume $w$ has minimal length in its conjugacy class in $W$. Then the
  following hold for any $v \in W$:
  \begin{enumerate}
  \item $J(w^v) = \DD(v^{-1}) \iff v = w_{J(w)}$;
  \item $J(w^v) = \AA(v^{-1}) \iff v = w_{J(w)} w_0$.
  \end{enumerate}
\end{Theorem}

\begin{proof}
  Let $L = \DD(v^{-1})$.  Then $v^{-1} = w_L \cdot x$ for some $x \in
  X_L$, by Lemma~\ref{1.5.2}.  Clearly $\ell((w^v)^{w_L}) =
  \ell(w^v)$, since $J(w^v) = L$.  By Lemma~\ref{2.1.14}, conjugation
  by the coset representative $x$ does not decrease the length, hence
  $\ell(w^v) = \ell((w^v)^{w_L}) \leq \ell((w^v)^{w_L x}) = \ell(w)$
  and it follows that $w^v$ has minimal length in its conjugacy class
  as well. By Proposition~\ref{3.1.11}, $L = J(w^v)$ and $J(w)$ are
  conjugate subsets of~$S$.  Proposition \ref{prop:1} (iv) says
  more precisely that $x$ is a conjugating element, i.e., $L^x = J(w)$.

  Assume that $\ell(x) > 0$ and let $s \in \DD(x)$. Then $s \notin L$,
  since $x \in X_L$; denote $L \cup \{s\}$ by $M$. By
  Theorem~\ref{2.3.3}, $x$ is a reduced product $x = d \cdot y$ with
  $y \in X_M$ and $d = w_L w_M$, the longest coset representative of
  $W_L$ in $W_M$. It follows that $v^{-1} = w_L \cdot x = w_M \cdot
  y$, whence $M \subseteq \DD(v^{-1}) = L \subsetneq M$.  The
  contradiction shows that $x = 1$, and therefore $v = w_L$ and $L =
  J(w)$.

  (ii)   Note  that   $J(w^{w_0})  =   J(w)^{w_0}$,   $\DD(x^{w_0})  =
  \DD(x)^{w_0}$, and $\AA(x)  = \DD(x w_0)$ for all  $x \in W$, whence
  $\DD(x)^{w_0} = \AA(w_0 x)$.  Therefore, it follows from (i) that
  \begin{align*}
    J(w^{x w_0}) = J(w^x)^{w_0} = \DD(x^{-1})^{w_0} = \AA((x
    w_0)^{-1}) \iff x = w_{J(w)},
  \end{align*}
  as desired, for $v = x w_0$.
\end{proof}

The following formula, first proved by Solomon~\cite{Solomon1966}
in 1966, is an easy consequence of the previous result.

\begin{Theorem}[Solomon's theorem]
  For $J \subseteq S$, let $\pi_J$ denote the permutation character of
  the action of $W$ on the cosets of $W_J$ defined by $\pi_J(w) =
  \Size{\Fix_{W/W_J}(w)}$, and let $\epsilon$ be the sign character of
  $W$, defined by $\epsilon(w) = (-1)^{\ell(w)}$ for $w \in W$. Then
  \begin{align*}
    \sum_{J \subseteq
    S} (-1)^{\Size{J}} \pi_J = \epsilon.
  \end{align*}
\end{Theorem}

\begin{proof}
  The formula follows if we can show that $\sum_{J \subseteq S} (-1)^{\Size{J}} \pi_J(w) =
  \epsilon(w)$ for all $w \in W$. 
  We have $W_J x w = W_J x \iff x w x^{-1} \in W_J$, so
  \begin{align*}
    \sum_{J \subseteq S} (-1)^{\Size{J}} \pi_J(w) 
    = \sum_{J \subseteq S} (-1)^{\Size{J}} \sum_{\substack{x \in X_J\\ xwx^{-1} \in W_J}} 1 
    = \sum_{x \in W} \sum_{J(xwx^{-1}) \subseteq J \subseteq \AA(x)} (-1)^{\Size{J}},
\end{align*}
where we reversed the order of summation and used the facts that $x
\in X_J \iff J \subseteq \AA(x)$ and $xw x^{-1} \in W_J \iff J(x w x^{-1})
\subseteq J$.  The Binomial Theorem implies that
\begin{align*}
  \sum_{A \subseteq J \subseteq B} (-1)^{\Size{J}} = (-1)^{\Size{A}} \sum_{I \subseteq B \setminus A} (-1)^{\Size{I}} = 
  \begin{cases}
    (-1)^{\Size{A}} & \text{if } A = B, \\ 
    0 &  \text{otherwise.} 
  \end{cases}
\end{align*}
But then
\begin{align*}
 \sum_{J \subseteq S} (-1)^{\Size{J}} \pi_J(w) = \sum_{\substack{x \in W \\ J(xwx^{-1}) =\AA(x)}} (-1)^{\Size{\AA(x)}}.
\end{align*}
Since $\sum_{J \subseteq S} (-1)^{\Size{J}} \pi_J$ is a class function, it is enough to choose 
 $w$ with minimal length in its conjugacy class. By Theorem~\ref{thm:3}(ii), this sum then 
 consists only of the one term for $x^{-1} = w_{J(w)} w_0$, and we have
\begin{align*} 
\sum_{J \subseteq S} (-1)^{\Size{J}} \pi_J(w) = (-1)^{\Size{\AA(w_0 w_{J(w)})}}.
\end{align*}
But $\Size{\AA(w_0 w_{J(w)})}
= \Size{\DD(w_{J(w)})^{w_0}}
= \Size{\DD(w_{J(w)})} =
\Size{J(w)}$ and then the claim follows from  the fact that $(-1)^{\Size{J(w)}} =
(-1)^{\ell(w)}$~\cite[Exercise 3.17]{GePf2000}.
\end{proof}

Solomon proves this formula generically for all types of finite
Coxeter groups, and he has published three different versions of the proof.  His
original proof~\cite{Solomon1966} depends on an application the Hopf
trace formula to the Coxeter complex of the finite Coxeter group $W$,
a later proof (of a more general statement) uses a decomposition of
the group algebra of $W$.  The third version of the
proof~\cite{Solomon1976} is based on properties of a homomorphism of
the descent algebra of $W$ into the character ring of~$W$ (see also
\cite[Exercise 3.15]{GePf2000}).  None of these proofs have the
combinatorial flavor of the above proof.

Finally, let us explain how to interpret Solomon's theorem as a generalization 
of a special case of the celebrated MacMahon master theorem. For a connection with a
different result due to MacMahon, see \cite[\S 6]{Solomon1968}.

MacMahon's master theorem states that for a matrix $X = (x_{ij})_{n\times n}$, 
the functions
\begin{equation} \label{macmahon}
 \frac 1{\det(\mathrm{Id}-X)},
\end{equation}
where $\mathrm{Id}$ denotes the $n \times n$-identity matrix, and
\begin{align*}
  \sum_{w} x_{v_1 w_1} x_{v_2 w_2} \cdots x_{v_m w_m},
\end{align*}
where $w = w_1 w_2 \cdots w_m$ 
runs over all words in $\{1,2,\ldots,n\}$ and $v = v_1 v_2 \cdots v_m$ is the
weakly increasing rearrangement of $w$, are equal. In particular, for
any permutation $w \in S_n$, the coefficient of $x_{1w(1)}\cdots x_{n
w(n)}$ in \eqref{macmahon} is equal to $1$. For $I \subseteq [n]$, denote by
$X_I$ the submatrix $(x_{ij})_{i,j \in I}$.  We have
\begin{multline*}
  \frac 1{\det(\mathrm{Id}-X)} = \frac 1{\sum_{I \subseteq [n]} (-1)^{\Size{I}} \det
X_I} = \frac 1{1 - \sum_{\emptyset \neq I \subseteq [n]} (-1)^{\Size{I}-1}
\det X_I} \\
= \sum_{k \geq 0} \Bigl( \sum_{\emptyset \neq I \subseteq
[n]} (-1)^{\Size{I}-1} \det X_I\Bigr)^k = \sum_{k \geq 0} \sum
(-1)^{\Size{I_1}-1 + \ldots + \Size{I_k}-1} \det X_{I_1} \cdots \det X_{I_k},
\end{multline*}
where the last sum runs over all $k$-tuples $(I_1,\ldots,I_k)$ of
non-empty subsets of $[n]$. Since we are interested in the coefficient
of $x_{1w(1)}\cdots x_{n w(n)}$ (in which all indices are represented,
and each index is represented only twice, once as a first index and
once as a second index), we can limit the sum to ordered set
partitions $(I_1,\ldots,I_k)$ of the set $[n]$. Note that we have
$(-1)^{\Size{I_1}-1 + \ldots + \Size{I_k}-1} = (-1)^{n-k}$.

Recall that the symmetric group $S_n$ is a Coxeter group $W$ of type $A_{n-1}$ with Coxeter generators
$S = \{s_1,\ldots,s_{n-1}\}$, $s_i = (i,i+1)$. Choose a composition $\lambda \vdash n$.
By Merris-Watkins formula~\cite{MW85} (and not hard to prove independently), 
the coefficient of $x_{1w(1)}\cdots x_{n w(n)}$ in
\begin{align*}
  \sum \det X_{I_1} \cdots \det X_{I_k},
\end{align*}
where the sum runs over all ordered set partitions $(I_1,\ldots,I_k)$
of $[n]$ with $\Size{I_j} = \lambda_j$ for all $j$, is equal to
$(-1)^{\ell(w)} \pi_J(w)$, where $J$ is the subset of $S$ that corresponds
to the composition $\lambda$. This means that
\begin{align*}
  \sum (-1)^{n - \Size{\lambda}} (-1)^{\ell(w)} \pi_J(w) = 1,
\end{align*}
where the sum runs over all subsets $J$ of $S$. This is obviously 
equivalent to Solomon's theorem.

\subsection*{Acknowledgement}
The second author wishes to acknowledge support from Science
Foundation Ireland.

\bibliography{centrali}
\bibliographystyle{amsplain}

\end{document}